\documentclass[journal]{IEEEtran}

\usepackage{paralist}
\ifCLASSINFOpdf
   \usepackage[pdftex]{graphicx}
\else
   \usepackage[dvips]{graphicx}
\fi
\usepackage{amsmath,amssymb,amsfonts,amsthm}
\ifCLASSOPTIONcompsoc
  \usepackage[caption=false,font=normalsize,labelfont=sf,textfont=sf]{subfig}
\else
  \usepackage[caption=false,font=footnotesize]{subfig}
\fi
\usepackage{url}

\usepackage{algorithm,algorithmic}
\usepackage{enumitem}
\usepackage{afterpage}
\allowdisplaybreaks

\newtheorem{thm}{Theorem}[section]
\newtheorem{rmk}{Remark}[section]
\newtheorem{assump}{Assumption}[section]
\newtheorem{alg}{Algorithm}[section]

\begin{document}
%
\title{An Adaptive Pole-Matching Method for Interpolating Reduced-Order Models}
%
%
%

\author{Yao~Yue,~\IEEEmembership{} 
        Lihong~Feng,~\IEEEmembership{} 
        and~Peter~Benner,~\IEEEmembership{}
        \thanks{Y. Yue, L. Feng and P. Benner are with Max Planck Institute for Dynamics of Complex Technical Systems, Magdeburg}
        \thanks{Manuscript submitted on July 30, 2019.}%
\thanks{This paper is an expanded version from
the IEEE MTT-S International Conference on Numerical Electromagnetic and Multiphysics Modeling and Optimization,
Boston, MA, USA, May 29-31, 2019.}\thanks{\textcopyright2019 IEEE.  Personal use of this material is permitted.  Permission from IEEE must be obtained for all other uses, in any current or future media, including reprinting/republishing this material for advertising or promotional purposes, creating new collective works, for resale or redistribution to servers or lists, or reuse of any copyrighted component of this work in other works.}}

\maketitle

\begin{abstract}
  An adaptive parametric reduced-order modeling method based on interpolating poles of reduced-order models is proposed in this paper.
  To guarantee correct interpolation, a pole-matching process is conducted to determine
  which poles of two reduced-order models correspond to the same parametric pole.
  First, the pole-matching in the scenario of parameter perturbation is discussed.
  It is formulated as a combinatorial optimization problem and solved by a branch and bound algorithm.
  Then, an adaptive framework is proposed to build repository ROMs at adaptively chosen parameter values, which well represent the parameter domain of interest.
  To achieve this, we propose techniques including a predictor-corrector strategy and an adaptive refinement strategy, which enable us to use larger steps to explore the parameter domain of interest with good accuracy.
The framework also consists of regression as an optional post-processing phase to further reduce the data storage.
  The advantages over other parametric reduced-order modeling approaches are,
  e.g., compatibility with any model order reduction method, constant size of the parametric reduced-order model with respect to the number of parameters, and capability to deal with complicated parameter dependency.
\end{abstract}

\begin{IEEEkeywords}
  numerical acceleration techniques, reduced order modeling, model-order reduction, frequency domain techniques, spectrum analysis
\end{IEEEkeywords}

%
\IEEEpeerreviewmaketitle

\section{Introduction}
%
%
%
%
\IEEEPARstart{P}{arametric} model order reduction (PMOR)~\cite{morBenGW15} has become a popular tool in the last decades to drastically reduce the computational cost of model-based parametric studies in many fields in engineering, e.g., circuit simulations~\cite{morFenYBetal16}, structural mechanics~\cite{YYKM11,SMDE14} and chromatography~\cite{morZhaFLetal14}.

Among the many methods proposed, projection-based PMOR methods have received the most research attention~\cite{morRozHP08,morBauBBetal11, morBenGW15}. Most of these methods first build a global basis by collecting data/derivative information at different parameter values, e.g., snapshots for time-domain systems~\cite{morRozHP08} and (cross-)moments in frequency-domain systems~\cite{morBauBBetal11}. Then, these methods project the parametric full-order model~(FOM) onto the subspace spanned by the basis to compute a parametric reduced-order model (pROM). However, these methods have several disadvantages. First, they assume that a parametric FOM is available in the state-space form, which is not always possible in industrial applications. For example, many engineering programs can only compute a FOM from the original partial differential equations at a given parameter value and they may use different meshes for discretization at different parameter values. In addition, they tend to lose efficiency when the number of parameters increases because the dimension of the subspace often increases fast with the number of parameters.

Recently, there are some research efforts that build pROMs by interpolating local matrices or bases.
In~\cite{morAmsF08}, it was proposed to interpolate the basis matrices on the geodesic of the Stiefel manifold.
However, the costs of both data storage and computation are high and according to our experience in~\cite{YFBE19}, it is difficult to achieve high accuracy with this approach.
Another line of research is to conduct PMOR via the interpolation of local matrices.
As is shown in~\cite{morPanMEetal10}, directly interpolating the local reduced-order state-space matrices does not work in general since the ROMs are usually in different state-space coordinate systems.
Therefore, it was proposed in~\cite{morPanMEetal10} to compute a global basis from the local bases and to transform all ROMs to a consistent state-space coordinate system defined by the global basis.
Another approach was proposed in~\cite{morAmsF11} to apply a congruence transformation to obtain ``consistent bases'' in a heuristic sense, which is done by solving a Procrustes optimization problem.
It has also been proposed to interpolate the local frequency response functions (FRFs)~\cite{morBauB09,morBauBGetal11}.
However, as we showed in~\cite{YFBE19}, these methods do not work in general.
For a comprehensive review of these methods, we refer to~\cite{YFBE19}.

In~\cite{YueFB18,YFBE19}, we proposed a pole-matching based parametric reduced-order modeling method, which builds a pROM by interpolating given nonparametric reduced-order models (ROMs).
It interpolates the positions and residues of all poles of the local ROMs, which has a clear physical meaning and removes all additional degrees of freedom introduced by realization.
This method only requires these ROMs to be accurate enough, no matter by what methods they are built. We showed that this method can even interpolate a ROM built by  a projection-based method and a ROM built by a data-driven method. To match the poles between different ROMs, we proposed to solve a combinatorial optimization problem. However, we did not provide an efficient algorithm to solve this optimization problem: using a brutal-force method quickly becomes too expensive computationally. Furthermore, misleading matching results may occur if the pole distributions of the two ROMs are very different.

In this paper, we improve the efficiency and robustness of the pole-matching based parametric reduced-order modeling method using the following techniques:
\begin{enumerate}
\item The branch and bound method for pole-matching. To guarantee correct pole-matching, we propose to use the optimization approach only in the case of pole perturbation, i.e., the corresponding poles of the two ROMs only differ slightly. The proposed branch and bound method is very efficient for this scenario because most enumeration cases are normally branched out thanks to the good initial guess.
\item An adaptive framework for parametric reduced-order modeling. This framework serves as an efficient and accurate way to select the parameter values adaptively, at which we build repository ROMs in order to represent the parameter range of interest well. A predictor-corrector technique and an adaptive refinement strategy are proposed to explore the parameter domain of interest using larger-steps with good accuracy. A regression technique is also proposed as a post-processing phase to further reduce the data storage.
\end{enumerate}

This paper is extended from our conference paper~\cite{YueFB19}.
Compared to~\cite{YueFB19}, we propose a new adaptive parametric reduced-order modeling framework, which consists of new techniques like adaptive refinement and regression. The new framework is more accurate and more robust.
Our numerical results show that the new method not only improves the relative error from a magnitude of $10^{-2}$ to a magnitude of $10^{-4}$, but also saves the storage of the pROM. Last but not least, we discuss the branch and bound method in more detail.

This paper is constructed as follows. In Section~\ref{sec:background}, we briefly review the results in~\cite{YueFB18}. Then in Section~\ref{sec:bb}, we discuss matching of perturbed poles and propose the branch and bound method. Based on this, a framework to adaptively build a pROM is proposed in~\ref{sec:continuation}. We end the paper with the numerical results in Section~\ref{sec:num_results} and conclusions in Section~\ref{sec:con}.

\section{Background: the Pole-Matching PMOR Method}\label{sec:background}
Assume that we have a series of ROMs and denote the ROM built for the parameter value $p_i$ ($i=1,2,\ldots,n_p$) by $(A^{(i)},B^{(i)},C^{(i)})$: 
\begin{align}
  \big( s I - A^{(i)} \big) x^{(i)}(s)&=B^{(i)} u(s), \nonumber \\ \label{eq:para_ROM}
  y(s)&=C^{(i)} x^{(i)}(s),
\end{align}
where $A^{(i)}\in\mathbb{R}^{k \times k}$, $B^{(i)}\in\mathbb{R}^{k \times 1}$, and $C^{(i)}\in\mathbb{R}^{1 \times k}$. Assume further that for each $i$, all eigenvalues of $A^{(i)}$ are simple.

\subsection{The Pole-Residue Realization}
Since a ROM has infinitely many equivalent realizations, the simple idea of interpolating $A^{(i)}$, $B^{(i)}$ and $C^{(i)}$ to obtain a pROM normally does not work~\cite{morPanMEetal10}. In~\cite{YueFB18}, we proposed to convert all ROMs to a unified pole-residue realization
\begin{align}
  \big( s I - \Lambda^{(i)} \big) x^{(i)}(s)&=\overline{B}^{(i)} u(s), \nonumber \\ \label{eq:mmr_ROM}
  y(s)&=\overline{C}^{(i)} x^{(i)}(s),
\end{align}
where using the notation $M^{\alpha,{T}}_\beta$ for $(M_\beta^\alpha)^{T}$, we define
\begin{align}\label{eq:mmr_diag}
  \Lambda^{(i)}&=\text{diag}\left[\Lambda_1^{(i)}, \Lambda_2^{(i)}, \ldots, \Lambda_m^{(i)}\right],\\
  \overline{B}^{(i)}&=\left[\overline{B}^{(i),T}_1,\overline{B}^{(i),T}_2,\ldots,\overline{B}^{(i),T}_m\right]^T,\\
  \overline{C}^{(i)}&=\left[\overline{C}^{(i)}_1,\overline{C}^{(i)}_2,\ldots,\overline{C}^{(i)}_m\right].
\end{align}
The pole-residue realization is closely related to the rational form of the transfer function $H^{(i)}(s)=\frac{y(s)}{u(s)}$ of~\eqref{eq:para_ROM}:
\begin{equation}
  H^{(i)}(s)
  =\sum_{j=1}^{n_s}\frac{\overline{C}^{(i)}_j}{s-\lambda^{(i)}_j}+\sum_{j=1}^{n_d}\frac{\overline{C}^{(i)}_{j,1}(s-a^{(i)}_j)-\overline{C}^{(i)}_{j,2}b^{(i)}_j}{(s-a^{(i)}_j)^2+(b^{(i)}_j)^2},
\end{equation}
where for a real eigenvalue $\lambda_j^{(i)}$, we have $\Lambda_j^{(i)}=\lambda_j^{(i)}$, $\overline{B}^{(i)}_j=1$ and $\overline{C}^{(i)}_j\in\mathbb{R}$, which corresponds to the term
\begin{equation}
  \frac{\overline{C}^{(i)}_j}{s-\lambda^{(i)}_j} = \overline{C}^{(i)}_j (s-\lambda^{(i)}_j)^{-1} \overline{B}^{(i)}_j,
\end{equation}
and for a pair of conjugate complex eigenvalues $a_j \pm \imath b_j$, we have $\Lambda_j^{(i)}\in\mathbb{R}^{2\times2},\overline{B}^{(i)}_j\in\mathbb{R}^{2\times 1},\overline{C}^{(i)}_j\in\mathbb{R}^{1\times 2}$, which corresponds to the term
\begin{align*}
  &\frac{\overline{C}^{(i)}_{j,1}(s-a^{(i)}_j)-\overline{C}^{(i)}_{j,2}b^{(i)}_j}{(s-a^{(i)}_j)^2+(b^{(i)}_j)^2} \\
    =& \left[\overline{C}^{(i)}_{j,1},\,\, \overline{C}^{(i)}_{j,2}\right] \left(
     sI - \left[
       \begin{array}{cc}
         a^{(i)}_j& b^{(i)}_j\\
         -b^{(i)}_j& a^{(i)}_j
       \end{array}
       \right]\right)^{-1} \left[ 
       \begin{array}{c}
         1 \\ 0
       \end{array}
       \right] \\[3ex]
     \stackrel{\triangle}{=}& \, \overline{C}^{(i)}_j (s-\Lambda^{(i)}_j)^{-1} \overline{B}^{(i)}_j.
\end{align*}

Therefore, the positions of the poles are stored in $\Lambda^{(i)}$ and the residues of poles are stored in $\overline{C}^{(i)}$. The matrix $\overline{B}^{(i)}$ contains no useful information: $\overline{B}^{(i)}_j\equiv 1$ for a real pole and $\overline{B}^{(i)}_j \equiv [1, 0]^T$ for a pair of conjugate complex poles.

For more details, we refer to~\cite{YueFB18,YFBE19}. 

\subsection{Efficient Storage}

Assuming that $A^{(i)}$ has $n_s$ real eigenvalues and
$n_d$ pairs of conjugate complex eigenvalues, we can store the ROM in the pole-residue realization with two matrices~\cite{YueFB18}:
\begin{align}\label{eq:mmr_storage}
  D_i \in \mathbb{C}^{n_d \times 4} \quad \text{and} \quad S_i \in \mathbb{C}^{n_s\times 2}.
\end{align}
The $j$-th row of $S_i$ is $(\lambda^{(i)}_j,\overline{C}^{(i)}_j)$, which corresponds to the real pole 
$  \frac{\overline{C}^{(i)}_j}{s-\lambda^{(i)}_j},$
while the $j$-th row of $D_i$ is $(a^{(i)}_j,b^{(i)}_j,\overline{C}^{(i)}_{j,1},\overline{C}^{(i)}_{j,2})$, which corresponds to a pair of conjugate complex poles
$  \frac{\overline{C}^{(i)}_{j,1}(s-a^{(i)}_j)-\overline{C}^{(i)}_{j,2}b^{(i)}_j}{(s-a^{(i)}_j)^2+(b^{(i)}_j)^2} $.
The order of the rows in $D_i$ and $S_i$ depends on their order in $\Lambda^{(i)}$: the order of rows in $S_i$ ($D_i$) is the order of real (complex conjugate) poles in $\Lambda^{(i)}$.

\section{Matching of Perturbed Poles:\newline a Branch and Bound Method}\label{sec:bb}
To interpolate two sufficiently accurate ROMs, say $(D_1,S_1)$ and $(D_2,S_2)$, to obtain a pROM, we first determine which poles in $(D_1,S_1)$ and $(D_2,S_2)$ correspond to the same parametric pole. This is what we call pole-matching. The order of the poles are influenced by: 
\begin{inparaenum}
\item the MOR algorithm;
\item the eigenvalue decomposition algorithm~\cite{matrix_comp96} that we use to compute the pole-residue realization~\eqref{eq:mmr_ROM}.
\end{inparaenum}
In this section, we discuss matching of perturbed poles under Assumption~\ref{assump:bb}. The matching of distant poles will be the topic of Section~\ref{sec:continuation}.

\begin{assump}\label{assump:bb} The assumptions for the branch and bound method.
  \begin{enumerate}
\item The sizes of $D_1$ and $S_1$ equal those of $D_2$ and $S_2$, respectively;
\item All poles of the ROM $(D_2,S_2,)$ are perturbations of poles of ROM $(D_1,S_1)$, but for a pair of perturbed poles, their order in the $D$/$S$ vectors may be different. More precisely, $r(\text{ROM}_1, \text{ROM}_2)$ is small, with $r$ defined later in~\eqref{eq:rom_dist}.
\end{enumerate}
\end{assump}

Under Assumption~\ref{assump:bb}, pole-matching can be done by solving an optimization problem, i.e., find an optimal way to rearrange the rows of $(D_2,S_2)$ so that it is the closest to $(D_1,S_1)$ in some sense.

\subsection{Matching of Perturbed Poles via Optimization}
We first define two mapping vectors: $v_d=[v_{d,1},\allowbreak v_{d,2},\allowbreak \ldots,\allowbreak v_{d,n_d}]$ for conjugate complex poles and $v_s=[v_{s,1},v_{s,2},\ldots,v_{s,n_s}]$ for real poles, which satisfy
\begin{align}
  \{v_{d,1}\} \cup \{v_{d,2}\} \cup \cdots \cup \{v_{d,n_d}\}=\{1,2,\ldots,n_d\},&&\label{eq:constr_vd} 
\end{align}
and
\begin{align}
  \{v_{s,1}\} \cup \{v_{s,2}\} \cup \cdots \cup \{v_{s,n_s}\}=\{1,2,\ldots,n_s\}.&&\label{eq:constr_vs}
\end{align}
Then we define the mapping $\mathcal{M}$, which rearranges the rows of $D$ (or $S$) according to the index vector $v_d$ (or $v_s$):
\begin{align}\label{eq:M_vd_defi}
  \mathcal{M}(v_d;D)&=[d_{v_{d,1}}^T, d_{v_{d,2}}^T, \ldots, d_{v_{d,n_d}}^T]^T,\\
  \mathcal{M}(v_s;S)&=[s_{v_{s,1}}^T, s_{v_{s,2}}^T, \ldots, s_{v_{s,n_s}}^T]^T,\label{eq:M_vs_defi}
\end{align}
where $D=[d_1^T,d_2^T,\ldots,d_{n_d}^T]^T$ and $S=[s_1^T,s_2^T,\ldots,s_{n_s}^T]^T$.

In the case of pole perturbation, we achieve pole-matching by solving the following optimization problem
\begin{align}
  \min_{v_d,v_s} f(v_d,v_s)=&\left\| D_1W_d- \mathcal{M}(v_d,D_2)W_d\right\|_F^2 + \nonumber \\  
  &\left\| S_1W_s- \mathcal{M}(v_s,S_2)W_s\right\|_F^2, \label{eq:objective}
\end{align}
where $W_d=\text{diag}\{w_p,w_p,w_r,w_r\}$ and $W_s=\text{diag}\{w_p,w_r\}$ are weighting matrices. The scalars $w_p$ and $w_r$ are the weights for the positions and residues of the poles, respectively.
Now we give the motivation of the optimization problem.
Assume that all poles are simple and the perturbation is sufficiently small.
If all poles are correctly matched, the objective function is very small;
otherwise, the objective function must be much larger.

\subsection{A Branch and Bound Algorithm}
To solve the combinatorial optimization problem~\eqref{eq:objective}, the exhaustive search method requires $n_d ! \times n_s !$ enumerations, which is computationally feasible only for very small ROMs. In this section, we discuss efficient solution of~\eqref{eq:objective}.

First, we observe that the optimization problem~\eqref{eq:objective} can be decoupled into two independent optimization problems:
\begin{align}\label{eq:obj_vd}
   \min_{v_d} f_d(v_d)=\left\| D_1W_d- \mathcal{M}(v_d,D_2)W_d\right\|_F^2 
\end{align}
and
\begin{align}\label{eq:obj_vs}
  \min_{v_s} f_s(v_s)=\left\| S_1W_s- \mathcal{M}(v_s,S_2)W_s\right\|_F^2.
\end{align}
Their solutions $v_d^*$ and $v_s^*$ solve the original optimization problem~\eqref{eq:objective}:
\begin{align}
  \min_{v_d,v_s} f(v_d,v_s)&=\min_{v_d} f_d(v_d)+\min_{v_s} f_s(v_s),\\
  \text{arg} \min_{v_d,v_s} f(v_d,v_s) &= \Big( \text{arg}\min_{v_d} f_d(v_d),\, \text{arg}\min_{v_s} f_s(v_s) \Big).
\end{align}
Using this method alone, we reduce the computational cost of~\eqref{eq:objective} from $n_d ! \times n_s !$ enumerations to $n_d ! + n_s !$ enumerations.
In the following discussions, we focus on the solution of~\eqref{eq:obj_vd}.
The optimization problem~\eqref{eq:obj_vs} can be solved similarly.

The second technique that we use to further reduce the computational cost of the combinatorial optimization problem~\eqref{eq:objective} is the branch and bound method~\cite{CLAU99}.
It can drastically reduce this computational cost in the scenarios that are the most useful in pole-matching, i.e., \ref{sce:1} and \ref{sce:2} on page~\pageref{sce:1}.
The efficiency of the branch and bound method for pole matching relies on the following theorem.

\begin{thm}\label{thm:bb}
  Assume that $1\leq i,j \leq n_d$ are two given indices ($i \neq j$). If there exist two mapping vectors $v_{d}^{(1)}$ and $v_{d}^{(2)}$  satisfying~\eqref{eq:constr_vd} and
  \begin{align}
    v_{d,i}^{(1)}&=v_{d,j}^{(2)}, \quad v_{d,j}^{(1)}=v_{d,i}^{(2)}, \label{eq:thm_1}\\
    v_{d,k}^{(1)}&=v_{d,k}^{(2)}, \quad (\forall 1\leq k \leq n_d,\, k\neq i \text{ and } k\neq j)\label{eq:thm_2}\\
    f_d &\left(v_d^{(2)} \right) - f_d \left(v_d^{(1)} \right) = c > 0,\label{eq:thm_3}
  \end{align}
  then any mapping vector $v_d$ satisfying~\eqref{eq:constr_vd} and
  \begin{equation}\label{eq:vd_branch_out}
    v_{d,i}=v_{d,i}^{(2)} \quad \text{and} \quad v_{d,j}=v_{d,j}^{(2)}
  \end{equation}
cannot be the optimal solution of~\eqref{eq:obj_vd}.
\end{thm}
\begin{proof}
  Assume $v_d^+$ is the optimal solution of~\eqref{eq:obj_vd} satisfying~\eqref{eq:vd_branch_out}.
  We construct $v_d^-$ as
  \begin{align*}
    v_{d,i}^-&=v_{d,j}^+=v_{d,i}^{(1)}, && v_{d,j}^- = v_{d,i}^+=v_{d,j}^{(1)}, \\
    v_{d,k}^-&=v_{d,k}^+, &&(\forall 1\leq k \leq n_d,\, k\neq i \text{ and } k\neq j).
  \end{align*}
  Then,
  \begin{align*}
    f_d(v_d^+)-f_d(v_d^-)=&\left\| D_1W_d- \mathcal{M}(v_d^+,D_2)W_d\right\|_F^2 - \\
    &\left\| D_1W_d- \mathcal{M}(v_d^-,D_2)W_d\right\|_F^2 \\
    =& \left\| \left[
      \begin{array}{c}
        D_{1,i} \\ D_{1,j}
      \end{array}\right] W_d
      - \left[\begin{array}{c}
        D_{2,v_{d,i}^+} \\ D_{2,v_{d,j}^+} 
      \end{array}\right] W_d
      \right\|_F - \\
      &\left\| \left[
      \begin{array}{c}
        D_{1,i} \\ D_{1,j}
      \end{array}\right] W_d
      - \left[\begin{array}{c}
        D_{2,v_{d,i}^-} \\ D_{2,v_{d_j}^-}
      \end{array}\right] W_d
      \right\|_F\\
      =& \left\| \left[
      \begin{array}{c}
        D_{1,i} \\ D_{1,j}
      \end{array}\right] W_d
      - \left[\begin{array}{c}
        D_{2,v_{d,i}^{(2)}} \\ D_{2,v_{d,j}^{(2)}} 
      \end{array}\right] W_d
      \right\|_F - \\
      &\left\| \left[
      \begin{array}{c}
        D_{1,i} \\ D_{1,j}
      \end{array}\right] W_d
      - \left[\begin{array}{c}
        D_{2,v_{d,i}^{(1)}} \\ D_{2,v_{d_j}^{(1)}}
      \end{array}\right] W_d
      \right\|_F\\
      =& \left\| D_1W_d- \mathcal{M}(v_d^{(2)},D_2)W_d\right\|_F^2 - \\
      &\left\| D_1W_d- \mathcal{M}(v_d^{(1)},D_2)W_d\right\|_F^2 \\
      =&f_d \left(v_d^{(2)} \right) - f_d \left(v_d^{(1)} \right)
      = c > 0.
  \end{align*}
  Therefore, $$f_d(v_d^-)=f_d(v_d^+)-c<f_d(v_d^+),$$
  which contradicts the assumption that $v_d^+$ is the optimal solution of $\displaystyle\min_{v_d} f_d(v_d)$ defined in~\eqref{eq:obj_vd}.
\end{proof}

Based on Theorem~\ref{thm:bb}, a branch and bound algorithm is proposed with two motivations to save the computational cost:
\begin{enumerate}
\item Start at a hopefully good initial guess $v_d^{(0)}=(1,2,\ldots,n_d)$, which is computationally free.
\item Try to decrease $f_d$ by swapping any two entries of the current $v_d$ and use the result from Theorem~\ref{thm:bb} to branch out the swaps that we are certain to be unable to further decrease $f_d$.
\end{enumerate}

\begin{rmk} Remarks on the two motivations.
  \begin{enumerate}[label=\textup{\textbf{On Motivation \arabic*}.},align=left]
  \item Although $v_d^{(0)}=(1,2,\ldots,n_d)$ is simple, it is actually often a quite good initial guess under the following assumption:
    {
    \begin{assump}\label{assump:initial_guess}\upshape Conditions in favor of a good initial guess $v_d^{(0)}=(1,2,\ldots,n_d)$. 
      \begin{enumerate}
      \item We use the same MOR method to compute the ROMs;
      \item The MOR algorithm normally preserves the order of parametric eigenvalues when $\|p_1-p_2\|$ is small;
      \item Both ROMs are accurate enough and $\|p_1-p_2\|$ is small enough.
      \end{enumerate}    
    \end{assump}
    }
    Under this assumption, it is reasonable to presume that when the parameter is perturbed,
    \begin{inparaenum}[1)]
    \item The order of the parametric eigenvalues is normally preserved.
    \item Even when the order changes, normally only a small number of indices change at the same time.
    \end{inparaenum}
  \item Since we start from a ``hopefully good'' initial guess, it is quite likely that we fail to decrease $f_d$ by swapping entries of $v_d$.
    This case is actually ideal for us because the failed swap can be branched out and we never need to try it again according to Theorem~\ref{thm:bb}.
    If the initial guess happens to be the optimal solution, we actually just  need to branch out all possibilities, which only needs $n_d(n_d-1)$ enumerations instead of $n_d !$ enumerations.
  \end{enumerate}
\end{rmk}

To keep track of the branched out swaps, we use the matrix $F$ defined as follows:
\begin{equation}\label{defi:F}
  F(i,j)=\left\{
  \begin{array}{ll}
    0, & \text{\parbox[t]{.6\linewidth}{when any $v_d$ with $v_{d,i}=v^{(0)}_{d,j}$ and $v_{d,j}=v^{(0)}_{d,i}$ has been branched out;}}\\[.5ex]
    1, & \text{otherwise.}
  \end{array}
  \right. 
\end{equation}
More specifically, we initially set
\begin{align}
  F=  J_{n_d}- I_{n_d},
\end{align}
where $J_{n_d}, I_{n_d}\in \mathbb{R}^{n_d \times n_d}$ are the matrix of ones and the identity matrix, respectively.
This means that initially, no swap has been branched out: only swapping an index with itself is forbidden.
Note that when we swap the $i$-th and $j$-th entries of $v_d$, we actually swap the $v_{d,i}$-th and $v_{d,j}$-th rows of $D_2$,
or equivalently,  swap the $v_{d,i}$-th and $v_{d,j}$-th entries of $v_d^{(0)}$.

\begin{alg}\normalfont
    A Branch and Bound Method to Solve the Pole-Matching Optimization Problem~\eqref{eq:obj_vd} \label{alg:bb}\\[-2.8\baselineskip]
    \noindent\makebox[\linewidth]{\rule{\linewidth}{0.8pt}}\\[1.3\baselineskip]
     \noindent\makebox[\linewidth]{\rule{\linewidth}{0.4pt}}\vskip-2.8ex
     \begin{algorithmic}[1]
       \STATE Initialize $v_d=(1,2,\ldots,n_d)$, the current lowest objective value $b_l=f_d(v_d)$, and $F=  J_{n_d}- I_{n_d}$.
       \WHILE{$\|F\|_F\neq 0$}
       \FOR{$i=1$ \TO $n_d$}
       \FOR{$j=1$ \TO $n_d$}
       \IF{$ F(v_{d,i},v_{d,j}) =0$}
       \STATE \textbf{continue} (go to line 4)
       \ELSE
       \STATE Set $v_d^*=v_d$. 
       \STATE Swap the $i$-th and the $j$-th entry of $v_d^*$.
       \IF{$f(v_d^*,v_s^{(0)}) \geq b_l$}
       \STATE Update $F(v_{d,i},v_{d,j}) =0$.
       \ELSE
       \STATE Set $v_d=v_d^*$ and $b_l=f_d(v_d^*)$, and break both for loops (go to line 2).
       \ENDIF
       \ENDIF
       \ENDFOR
       \ENDFOR
       \ENDWHILE
     \end{algorithmic}\vskip-.9\baselineskip
      \noindent\makebox[\linewidth]{\rule{\linewidth}{0.8pt}}\vskip-.9\baselineskip
\end{alg}

Now we discuss the definition~\eqref{defi:F}. Assume that we are at $v_d^C$ and we have failed to increase $f_d$ by swapping its $i$-th entry with its $j$-th entry.
  Denote the swapped vector by $v_d^S$. Then $v_d^C$ and $v_d^S$ satisfy the conditions~\eqref{eq:thm_1}, \eqref{eq:thm_2} and~\eqref{eq:thm_3}.
  Therefore, according to Theorem~\ref{thm:bb}, any $v_d$ with $v_{d,i}=v_{d,i}^S=v_{d,j}^C$ and $v_{d,j}=v_{d,j}^S=v_{d,i}^C$ cannot be the solution of the optimization problem~\eqref{eq:obj_vd}. Therefore, we set $F(v_{d,i}^C,v_{d,j}^C)=0$ to forbid this swap, i.e., if the $k$-th and the $l$-th entries of any $v_d$ satisfy
  \begin{equation}\label{eq:vd_swap}
    v_{d,k}=v_{d,i}^C, \quad v_{d,l}=v_{d,j}^C,
  \end{equation}
we do not need to try the swap of the $k$-th and $l$-th entries of $v_d$ since it cannot decrease the objective $f_d$.
Note that the matrix $F$ is normally non-symmetric because if $F(v_{d,j},v_{d,i})=0$, we forbid swapping the $k$-th and the $l$-th entries if $v_{d,k}=v_{d,j}^C$ and $v_{d,l}=v_{d,i}^C$, which is actually a different swap.

As for implementation, we can also use the more compact ``potentially feasible tables\footnote{We use the term ``potentially feasible'' because the swaps recorded in this table might not be feasible, but we have not validated their feasibility yet.}'' $\mathcal{F}_1, \ldots, \mathcal{F}_{v_d}$ as we did in~\cite{YueFB18} instead of using the matrix $F$, especially when $n_d$ is large. The potentially feasible table is defined as
\begin{equation}
  \mathcal{F}_i=\{j| F(i,j)=1\}.
\end{equation}
So when a swap is branched out, it is simply removed from the potentially feasible table.

Now we discuss several scenarios in pole matching.
\begin{enumerate}[label=\textup{\textbf{Scenario \arabic*}.},ref={Scenario \arabic*},align=left]
\item\label{sce:1} All poles naturally match: the most optimistic scenario. This case occurs quite often under Assumption~\ref{assump:initial_guess} as we have discussed there. 
Using the branch and bound algorithm, we just need to verify that the initial $v_d^{(0)}$ is optimal and we only need to enumerate the $v_d$ vector $n_d(n_d-1)$ times instead of $n_d !$ times.
\item\label{sce:2}  All poles do not naturally match, but only a few swapping operations, say, $m$ times, are required to obtain the optimal solution. In general, this scenario is unavoidable. e.g., when the dominances\footnote{In this paper, we use the term ``pole dominance'' to denote the importance of a pole on the system output. We use it as a general term for all kinds of MOR methods. This term is borrowed from a specific case: the dominant pole algorithm~\cite{morMarLP96,SMDE14}, which defines the dominance of a pole of the form $\displaystyle \frac{\overline{C}^{(i)}_j}{s-\lambda^{(i)}_j}$ by $\displaystyle\frac{\big|{\overline{C}^{(i)}_j}\big|}{\big|\Re\big\{\lambda^{(i)}_j\big\}\big|}$ for this specific case.} of two poles cross. The computational cost in this scenario is bounded by $m n_d(n_d-1)$ enumerations. The actual computational cost is normally much lower because many swaps have been branched out at an early stage.
\item\label{sce:3}  Almost all poles between $(D_1,S_1)$ and $(D_2,S_2)$ do not naturally match. In this scenario, the computational cost of optimization is high. Actually, if both ROMs are generated by the same MOR method with some ``algorithmic continuity'', it is not so meaningful to conduct optimization, whose solution can be misleading. In such cases, it is highly possible that the ``perturbation'' is too large, except for some rare cases, e.g.,  the dominances of many pair of poles do cross at the same parameter value. Therefore, we avoid solving the optimization problem in this case, which is not only computationally expensive, but also useless. This is a motivation of the predictor-corrector method that will be presented in Section~\ref{sec:pred_corr}.
\end{enumerate}

\section{The Adaptive Framework to Build the Parametric ROM}\label{sec:continuation}

The goal of this section is to adaptively build a series of repository ROMs, which will be used as interpolates for constructing the pROM.
In this section, we confine our discussion to parametric systems with a single parameter $p\in \mathbb{R}$.
We denote these repository ROMs, whose poles have already been matched to those of all other repository ROMs, by $\overline{\text{ROM}}_1$, $\overline{\text{ROM}}_2$, \ldots, $\overline{\text{ROM}}_N$ built for the parameter values $p_1 <p_2 < \ldots <p_N$, respectively. In this paper, the overline symbol $\overline{\cdot}$ indicates the name or variables of a repository ROM, e.g., $\overline{\text{ROM}}_i$ with its pole-residue realization $(\overline{D}_i,\overline{S}_i)$.

For ease of notation, we introduce a mapping for pole-matching $\mathcal{P}_{D_1}(D_2)$ ($\mathcal{P}_{S_1}(S_2)$). When we match $D_2$ to $D_1$, we denote the matching result of Algorithm~\ref{alg:bb} by $\mathcal{P}_{D_1}(D_2)$. 
 Using this notation, we define the distance between ROM$_1$ $(D_1,S_1)$ and ROM$_2$ $(D_2,S_2)$ as
\begin{align}\label{eq:rom_dist}
  r(\text{ROM}_1, \text{ROM}_2)=&\big\| D_1 W_d - \mathcal{P}_{D_1}(D_2) W_d\big\|_F+\nonumber \\
  &\big\|S_1 W_s - \mathcal{P}_{S_1}(S_2) W_s\big\|_F
\end{align}
and the relative error by
\begin{align}\label{eq:rom_error}
  e(\text{ROM}_1, \text{ROM}_2)=&\frac{\big\| D_1 W_d - \mathcal{P}_{D_1}(D_2)W_d\|_F}{\|D_1\|_F}+\nonumber \\
  &\frac{\big\|S_1 W_s - \mathcal{P}_{S_1}(S_2) W_s\big\|_F}{\|S_1\|_F}.
\end{align}
For convenience, we denote the pole-matched ROM $(\mathcal{P}_{D_1}(D_2),\mathcal{P}_{S_1}(S_2))$ by $\mathcal{P}_{\text{ROM}_1}(\text{ROM}_2)$.

\subsection{The Predictor-Corrector Strategy}\label{sec:pred_corr}
The two motivations of the predictor-corrector strategy are:
\begin{enumerate}[label=\textup{\textbf{Objective \arabic*}.},ref={Objective \arabic*},align=left]
\item\label{cont_obj:1} We want to design a framework, in which \ref{sce:1} occurs most often, \ref{sce:2} occurs sometimes, and \ref{sce:3} occurs rarely.
\item However, we do not want to achieve \ref{cont_obj:1} simply by using small steps in $p$, which is inefficient.
\end{enumerate}

In order to explore the parameter space adaptively, we make the following assumptions on the ROMs.
\begin{assump}\label{assump:accuracy}
  At any feasible parameter value $p$, a ROM of high accuracy can be built at request.
\end{assump}
\begin{assump}\label{assump:pole_change}
  The changes in positions and residues of all poles are small when the change of $p$ is sufficiently small.
\end{assump}
Suppose that we have already computed the repository ROMs: $\overline{\text{ROM}}_1$, $\overline{\text{ROM}}_2$, \ldots, $\overline{\text{ROM}}_i$ and we want to expand the series by adding new repository ROMs.
At a new parameter value $p_{i+1}$ ($p_{i+1}>p_i$), 
we first compute a candidate ROM$_{i+1}^\text{cand}$ and try to match its poles to the repository ROMs. Instead of applying Algorithm~\ref{alg:bb} directly to $\overline{\text{ROM}}_i$ and ROM$_{i+1}^\text{cand}$, we propose to use a predictor-corrector strategy.

We first use the positions and residues of  poles of $\overline{\text{ROM}}_1$, $\overline{\text{ROM}}_2$, \ldots, $\overline{\text{ROM}}_i$ to calculate a ``predicted ROM'' ROM$_{i+1}^\text{pred}$ for $p_{i+1}$. For prediction, we can use extrapolation, e.g., linear, polynomial, spline, etc., on the positions and residues of the poles. If the calculated poles exhibit some noise, we can also use regression. Normally, the poles of ROM$_{i+1}^\text{pred}$ are closer to the poles of ROM$_{i+1}^\text{cand}$ than those of $\overline{\text{ROM}}_i$. To be safe, we check whether $r(\overline{\text{ROM}}_i, \text{ROM}_{i+1}^\text{cand})>r(\text{ROM}_{i+1}^\text{pred},\text{ROM}_{i+1}^\text{cand})$ holds. If it holds, we set $\overline{\text{ROM}}_{i+1}=\mathcal{P}_{\text{ROM}_{i+1}^\text{pred}}(\text{ROM}_{i+1}^\text{cand})$; otherwise, we set $\overline{\text{ROM}}_{i+1}=\mathcal{P}_{\overline{\text{ROM}}_{i}}(\text{ROM}_{i+1}^\text{cand})$.

Note that although we have ``accepted'' the current $\overline{\text{ROM}}_{i+1}$ as a repository ROM, it is not guaranteed yet that the interpolation will give accurate pMOR on the interval $(p_i,p_{i+1})$. Actually, the pole-matching is not necessarily correct because Assumption~\ref{assump:bb} may not be satisfied. However, even if the poles are wrongly matched, we store it for further use because we do not want to waste the computational effort with which we built $\overline{\text{ROM}}_{i+1}$. Therefore, we introduce an index $i_h$ to denote the last index that we are sure to have ``high fidelity'', i.e., it is correctly matched and close enough to $\overline{\text{ROM}}_{i_h-1}$ to produce good interpolation results. Here we assign $i_h=i$, which means that the fidelity of $\overline{\text{ROM}}_{i+1}$ still needs to be checked. More details will be discussed in the next section.

\subsection{An Adaptive Refinement Strategy}
In the previous section, we have discussed the efficient generation and pole-matching of ROMs.
However, this alone does not guarantee good accuracy of the pROM: it is possible that the poles are matched correctly, but the interpolation result does not capture the parametric evolution of the poles well. For example, for a system that has only one pole, although the pole is naturally ``matched'', we still need a fine grid of $p$ to capture the dynamics of the system when the dynamics is complex.

After we have computed $\overline{\text{ROM}}_{i+1}$ using the predictor-corrector strategy, we need to assess whether good accuracy can be achieved in the interval $(p_i,p_{i+1})$ using interpolation. To achieve this, we choose the test point $p^t=\frac{p_i+p_{i+1}}{2}$ and build two ROMs at $p^t$:
\begin{enumerate}
\item $\text{ROM}^{I,t}$ computed by interpolating the repository ROMs.
\item $\text{ROM}^{T,t}$ computed by a MOR algorithm.
\end{enumerate}
Since under Assumption~\ref{assump:accuracy}, $\text{ROM}^{T,t}$ is highly accurate, we regards $\text{ROM}^{I,t}$ as accurate only when
\begin{equation}\label{eq:test_accuracy}
  e(\text{ROM}^{I,t},\text{ROM}^{T,t})<\tau_e,
\end{equation}
where $\tau_e$ is a small real number, e.g., 0.001.
The value of $\tau_e$ should take the error of the MOR algorithm into consideration: $\tau_e$ should be sufficiently larger than an estimated error of the MOR algorithm.
Otherwise, the relationship \eqref{eq:test_accuracy} may never be satisfied despite of refinement of the grid.
In practical computations, we can actually change $\tau_e$ adaptively according to an error bound of $\text{ROM}^{T,t}$ rather than using the same $\tau_e$ for all iterations.

If~\eqref{eq:test_accuracy} is met, we trust the fidelity of  $\overline{\text{ROM}}_{i+1}$ and use the predictor-corrector for the next step.
Otherwise, we need more ROMs to represent the interval $(p_i,p_{i+1})$.
A natural choice is $\text{ROM}^{T,t}$ since it has already been built. Therefore, we assign
\begin{equation}
  \overline{\text{ROM}}_{i+2} \leftarrow \overline{\text{ROM}}_{i+1}, \quad \overline{\text{ROM}}_{i+1} \leftarrow \mathcal{P}_{\overline{\text{ROM}}_i}(\text{ROM}^{T,t}).
\end{equation}
Then, we check the fidelity of the new $\overline{\text{ROM}}_{i+1}$ using the procedure described above:
\begin{itemize}
\item If~\eqref{eq:test_accuracy} is violated, we do a further refinement for the interval $(p_i,p_{i+1})$.
\item Otherwise, we accept that $\overline{\text{ROM}}_{i+1}$ has high fidelity, update $i_h=i+1$ and
  \begin{equation}
    \overline{\text{ROM}}_{i+2} \leftarrow \mathcal{P}_{\overline{\text{ROM}}_{i+1}}(\overline{\text{ROM}}_{i+2})
  \end{equation}
  because $p_{i+1}$ (the former $p^t$) is closer to $p_{i+2}$ (the former $p_{i+1}$) and the pole-matching result is more trustworthy.
  Next, we check whether $\overline{\text{ROM}}_{i+2}$ has high fidelity:
  \begin{itemize}
  \item If it has high fidelity, we set $i_h=i+2$ and end this local refinement.
  \item Otherwise, we do a further refinement for the interval $(p_{i+1},p_{i+2})$ using the procedure above.
  \end{itemize} 
\end{itemize}
For the whole procedure, we refer to Algorithm~\ref{alg:adaptive}.
\subsection{Compact Parametric ROM by Regression}
After we have finished the construction of $\overline{\text{ROM}}_1$, $\overline{\text{ROM}}_2$, \ldots, $\overline{\text{ROM}}_N$ for the whole interval of interest, we can further save data storage by using regression. In this paper, we do regression for each entry of $\overline{D}_i(j,k)$ (the $(j,k)$-th entry of $\overline{D}_i$ built at $p_i$) and $\overline{S}_i(j,k)$ to obtain the parametric forms $\overline{D}(p;j,k)$  and $\overline{S}(p;j,k)$. Using a polynomial of degree $q$ for each entry, we reduce the storage cost from $N(4n_d+2n_s)$ to $(q+1)(4n_d+2n_s)$ since we only need to store the coefficients of the polynomials. This reduction is usually drastic since a low-order polynomial is normally enough when the system dynamics is not too complex.

We summarize all the techniques discussed in Algorithm~\ref{alg:adaptive} with more details.
\begin{alg}\normalfont
  An Adaptive Pole-Matching PMOR method\label{alg:adaptive}
  \\[-1.8\baselineskip]
    \noindent\makebox[\linewidth]{\rule{\linewidth}{0.8pt}}\\[.2\baselineskip]
     \noindent\makebox[\linewidth]{\rule{\linewidth}{0.4pt}}\vskip-2.8ex
     \begin{algorithmic}[1]
       \STATE \textbf{Input}: the parameter range of interest $[p^L,  p^U]$.
       \STATE \textbf{Initialization}: initialize step length $u_0$, error tolerance $\tau_e$, order of regression polynomials $q$, $i=2$, $p_1=p^L$, and compute $\overline{\text{ROM}}_1$.
       \REPEAT
       \STATE \textbf{Phase 1: Predictor-Corrector}
       \STATE Set $p_{i}=p_{i-1} + u_0$. If $p_{i}>p^U$, set $p_{i}=p^U$.    
       \STATE Build ROM$_{i}^\text{cand}$ with a MOR algorithm.
       \IF{i=2}
       \STATE $\overline{\text{ROM}}_{i}=\mathcal{P}_{\overline{\text{ROM}}_{i-1}}(\text{ROM}_{i}^\text{cand})$.
       \ELSE
       \STATE Compute ROM$_{i}^\text{pred}$ using extrapolation or regression of $\overline{\text{ROM}}_j$ ($j=1,2,\ldots,i-1$).
       \IF{$r(\overline{\text{ROM}}_{i-1}, \text{ROM}_{i}^\text{cand})>r(\text{ROM}_{i}^\text{pred},\text{ROM}_{i}^\text{cand})$}
       \STATE $\overline{\text{ROM}}_{i}=\mathcal{P}_{\text{ROM}_{i}^\text{pred}}(\text{ROM}_{i}^\text{cand})$.
       \ELSE 
       \STATE $\overline{\text{ROM}}_{i}=\mathcal{P}_{\overline{\text{ROM}}_{i-1}}(\text{ROM}_{i}^\text{cand})$.
       \ENDIF
       \ENDIF
       \STATE \textbf{Phase 2: Adaptive Refinement}
       \STATE Set $i_h=i-1$.
       \REPEAT
       \STATE At $p^t=\frac{p_{i_h}+p_{i_h+1}}{2}$, build an interpolated ROM, namely $\text{ROM}^{I,t}$, and a true ROM, namely $\text{ROM}^{T,t}$.
       \IF{$e(\text{ROM}^{I,t},\text{ROM}^{T,t})<\tau_e$}
       \STATE $i_h=i_h+1$.
       \ELSE
       \STATE $i=i+1$.
       \FOR{$j=i-1$ \TO $i_h+1$}
       \STATE $\overline{\text{ROM}}_{j+1}\leftarrow \overline{\text{ROM}}_{j}$, $p_{j+1}\leftarrow p_j$.     
       \ENDFOR
       \STATE $\overline{\text{ROM}}_{i_h+1}=\mathcal{P}_{\overline{\text{ROM}}_{i_h}}(\text{ROM}^{T,t})$.
       \FOR{$j=i-1$ \TO $i_h+1$}\label{line:reorder}
       \STATE $\overline{\text{ROM}}_{j+1}\leftarrow \mathcal{P}_{\overline{\text{ROM}}_{i_h+1}}(\overline{\text{ROM}}_{j+1})$.
       \ENDFOR
       \ENDIF
       \UNTIL{$i_h=i$.}
       \UNTIL{$p_i=p^U$.}
       \STATE \textbf{Optional Post-Processing: Regression}
       \FOR{$u=1$ \TO $n_d$, $v=1$ \TO 4}
       \STATE Apply the regression algorithm to $\overline{D}_1(u,v)$, $\overline{D}_2(u,v)$, \ldots, $\overline{D}_{i}(u,v)$ to compute the coefficients of the polynomial $d^r_0(u,v)$, $d^r_1(u,v)$, \ldots, $d^r_q(u,v)$.
       \ENDFOR
       \FOR{$u=1$ \TO $n_s$, $v=1$ \TO 2}
       \STATE Apply the regression algorithm to $\overline{S}_1(u,v)$, $\overline{S}_2(u,v)$, \ldots, $\overline{S}_{i}(u,v)$ to compute the coefficients of the polynomial $s^r_0(u,v)$, $s^r_1(u,v)$, \ldots, $s^r_q(u,v)$.
       \ENDFOR
     \end{algorithmic}
     \vskip-.9\baselineskip
      \noindent\makebox[\linewidth]{\rule{\linewidth}{0.8pt}}\vskip-.9\baselineskip
\end{alg}
Note that when pole-matching becomes difficult, i.e., too many swaps take place, we can simply give it up, save the ROM (with pole unmatched) for the moment and insert another ROM for refinement.The pole-matching will be done automatically at a later stage. A further remark is that in line~\ref{line:reorder}, we can also conduct fewer pole-matchings, e.g., only for $i_h+1$.

\subsection{The Offline-Online Strategy}
The proposed PMOR method can be implemented in an offline-online manner.
\begin{itemize}
\item In the offline phase, we use Algorithm~\ref{alg:adaptive} to explore the parameter range of interest.
  It is relatively expensive: we need to adaptively choose the parameter values, at which we build ROMs, compute the ROMs, conduct pole matching, and optionally, perform regression.
\item In the online phase, we simply interpolate the repository ROMs or evaluate the polynomials computed from regression. The online phase is computationally very cheap.
\end{itemize}

\section{Numerical Results}\label{sec:num_results}
In this section, we use two numerical examples to test the adaptive pole-matching PMOR method.

\subsection{The Nonlinear Parametric ``\texttt{\textit{FOM}}'' Model}
This academic example is adapted from the parametric ``\texttt{\textit{FOM}}'' model in~\cite{morIonA14} to introduce more complex system dynamics such as nonlinear parametric dependency and pole crossing, since in the original model, the parameter dependency is linear\footnote{ The parametric ``\texttt{\textit{FOM}}'' model in~\cite{morIonA14} is adapted from the nonparametric ``\texttt{\textit{FOM}}'' model in~\cite{morChaV02}}.
The model is of the form
\begin{align*}
  (s\mathcal{I}-\mathcal{A}(p))X(s,U)&=\mathcal{B}U, \\
  Y&=\mathcal{C} X(s,U),
\end{align*}
where $\mathcal{C}=[100,100,100,100,100,100,100,100,1,\ldots,1]\in\mathbb{R}^{1\times 1008}$, $B=C^T$, and $\mathcal{A}=\text{diag}\{\mathcal{A}_1,\mathcal{A}_2,\mathcal{A}_3,\mathcal{A}_4\allowbreak,1,\allowbreak2,\allowbreak3,\allowbreak\ldots,1000\}\in\mathbb{R}^{1008 \times 1008}$ with
\begin{align}
  \mathcal{A}_1&=\left[
    \begin{array}{cc}
        4p-42 & 8p+200\\
        -8p-200 & 4p-42
    \end{array}
    \right], \nonumber\\
  \mathcal{A}_2&=\left[
    \begin{array}{cc}
        2p-50 & p^2+4p+210\\
        -p^2-4p-210 & 2p-50
    \end{array}
    \right], \nonumber\\
  \mathcal{A}_3&=\left[
    \begin{array}{cc}
        -25+p & 100+p^2\\
        -100-p^2 & -25+p
    \end{array}
    \right], \nonumber\\
  \mathcal{A}_4&=\left[
    \begin{array}{cc}
        -25+2p & 150-p^2\\
        -150+p^2 & -25+2p
    \end{array}
    \right]. \nonumber
\end{align}
In this model of order 1008, the parameter range of interest is $[-10,10]$.
In Algorithm~\ref{alg:adaptive}, we set the initial conditions $u_0=\frac{\pi}{3}$, $\tau_e=0.001$ and $q=5$.
For all figures, the relative error is computed as $\frac{\left|\int_1^{1000}  \mathcal{H}(\omega,p)-  H(\omega,p)\, \rm{d}\omega \right|}{\left|\int_1^{1000}  \mathcal{H}(\omega,p) \rm{d}\omega \right|}$.

\begin{figure}[htp]
  \centering
  \subfloat[The Poles of Repository ROMs]{\label{subfig:poles}\includegraphics[width=.49\linewidth]{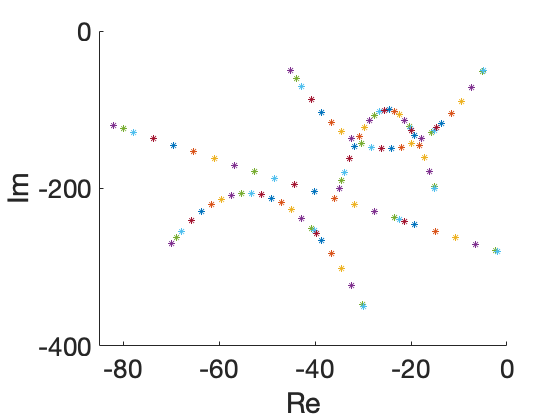}}
  \subfloat[The Relative Error]{\label{subfig:intp_error}\includegraphics[width=.49\linewidth]{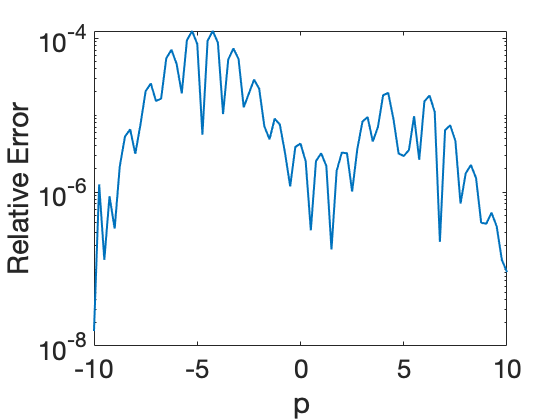}}\hskip0ex
  \caption{Interpolation Results}
  \label{fig:intp}
\end{figure}

\begin{figure}[htp]
  \centering
  \subfloat[The Relative Error]{\label{subfig:reg_error}\includegraphics[width=.49\linewidth]{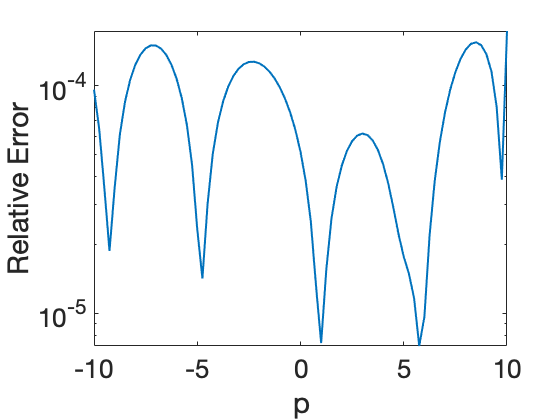}}\hskip0ex
  \subfloat[The Regressive Polynomials]{\label{subfig:reg_poles}\includegraphics[width=.49\linewidth]{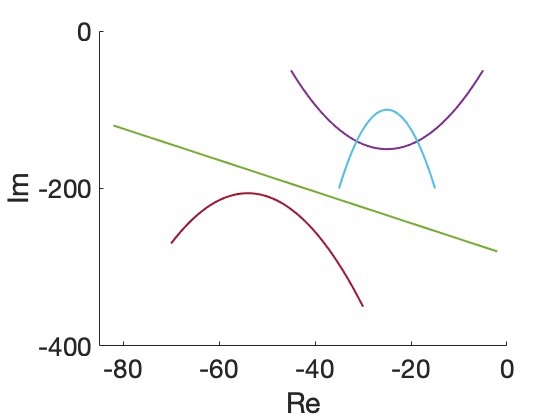}}
  \caption{Regression Results}
  \label{fig:regression}
\end{figure}

\begin{figure}[htp]
  \centering
  \subfloat[Interpolation Error]{\label{subfig:intp_error_wo_adaptive}\includegraphics[width=.49\linewidth]{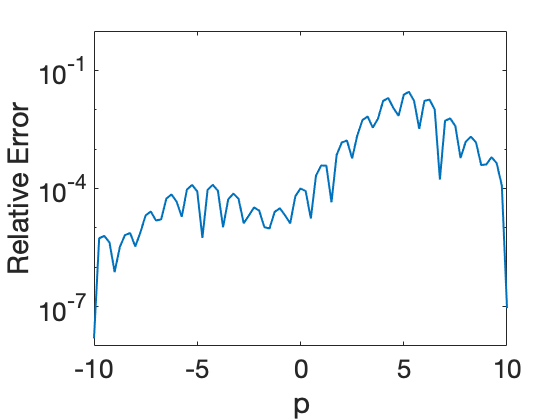}}\hskip0ex
  \subfloat[Regression Error]{\label{subfig:reg_error_wo_adaptive}\includegraphics[width=.49\linewidth]{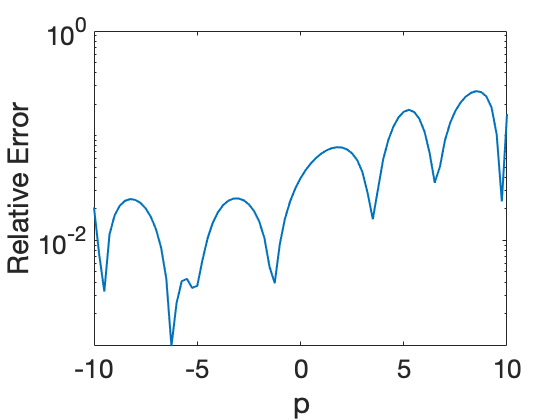}}\\
  \subfloat[The Regressive Polynomials]{\label{subfig:reg_poles_wo_adaptive}\includegraphics[width=.49\linewidth]{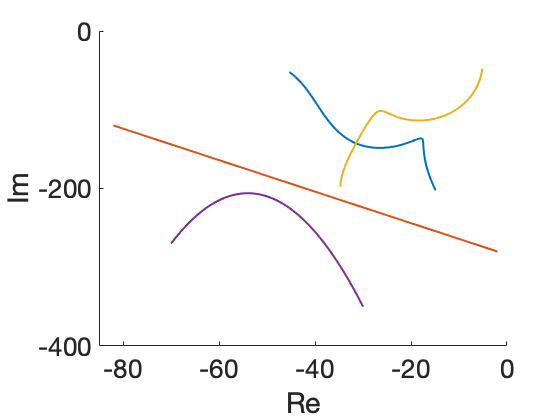}}
  \caption{Results without Adaptive Refinement}
  \label{fig:results_wo_adaptive}
\end{figure}

In Fig.~\ref{fig:intp}, we show the results of interpolation, i.e., when we skip the post-processing for regression and use the repository ROMs directly for interpolation. In total, 24 repository ROMs are built. In Fig.~\ref{fig:regression}, we show the numerical results when the post-processing for regression is also performed. We observe that using order-5 polynomials for regression results in pretty accurate results. The data storage is reduced by a factor of $\frac{24}{5}$. In Fig.~\ref{fig:results_wo_adaptive}, we plot the results when we skip the adaptive refinement. In this case, 21 repository ROMs are built and we observe that some poles are wrongly matched and the regression follows wrong branches. This shows the importance of the adaptive refinement in guaranteeing accuracy, especially when the step length is large. The much better numerical result using adaptive refinement is only at the cost of 3 more repository ROMs.

\subsection{The Branchline Coupler: a Microwave Device}
In this section, we apply our method to a branchline coupler model, which is a discretization of a time-harmonic Maxwell's equation. The full-order model is of the form
\begin{align}
  \left(\mathcal{K}(\mu)-s\mathcal{M}\right)X&=B,\nonumber\\
  Y&=CX,
\end{align}
where $\mathcal{K}(\mu)=\frac{1}{\mu}\mathcal{K}_0$, $\mathcal{K}_0,\mathcal{M}\in\mathbb{R}^{33051\times 33051}$, $B\in\mathbb{R}^{33051\times 1}$, $C\in\mathbb{R}^{1 \times 33051}$ and $s=\omega^2$. For detailed description of the model, we refer to~\cite{morwiki_branchcouple,morHesB13}. The repository ROMs of order 10 are built by a Krylov method~\cite{morAnt05}, in which we use the interpolating expansion points $s=1, 25, 50, 75, 100$ and order 2 for each of these interpolating expansion points. Since the reduced matrix of $\mathcal{K}(\mu_i)$ is nonsingular, the repository ROMs can be easily written into the form~\eqref{eq:para_ROM}.
To build a pROM for the parameter range $\mu\in[0.5,0.9]$, Algorithm~\ref{alg:adaptive} built 18 repository ROMs under the tolerance $\tau_e=0.00001$. Figure~\ref{fig:branchline} shows that both the interpolated pROM and the regressed pROM capture the dynamics of the system well. 

\begin{figure}[htp]
  \centering
  \subfloat[The Response of the FOM]{\label{subfig:branchline_frf}\includegraphics[width=.7\linewidth]{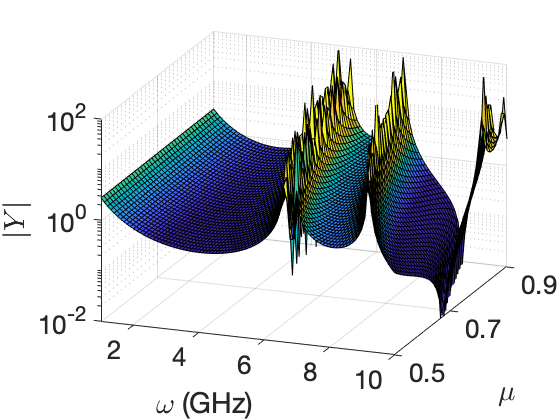}}\\
  \subfloat[The Relative Error of the Interpolated pROM]{\label{subfig:branchline_interp_err}\includegraphics[width=.7\linewidth]{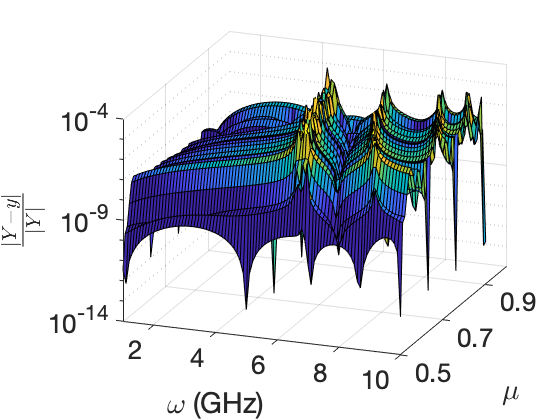}}\\
  \subfloat[The Relative Error of the Regressed pROM (order 7 polynomials are used for regression)]{\label{subfig:branchline_reg_err}\includegraphics[width=.7\linewidth]{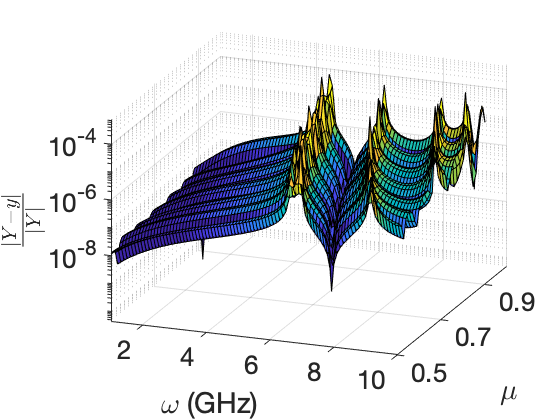}}  
  \caption{Numerical Results for the Branchline Example. Here, $Y$ and $y$ stand for the outputs of the FOM and the pROM, respectively.}
  \label{fig:branchline}
\end{figure}

\section{Further Discussions}
Here are some further discussions.
\begin{itemize}
\item The regression is successful only when all matchings are correct for dominant poles. The influence of mismatched poles on the interpolated pROM is only local, which is normally not too bad because the wrongly matched poles normally have similar positions and residues. For regression, however, the influence will be global because wrong branches are chosen. In this case, we actually try to regress a non-smooth function with polynomials, which is difficult. This phenomenon is shown in Fig.~\ref{fig:results_wo_adaptive}.
\item We would also like to emphasize that the interpolation approach is more general than the regression approach. The interpolation approach works well as long as the parameter dependence can be \emph{locally} captured by the basis functions, e.g., using polynomials for exponential functions or trigonometric functions. The regression method, on the other hand, works well only for the cases that the parameter dependence can be \emph{globally} captured by the basis functions. The interpolation approach even has potential to deal with eigenvalue bifurcation: even without any further considerations, the interpolation-based pROM is accurate as long as the parameter value is not too close to the bifurcation point. More detailed treatments of eigenvalue bifurcation will be future work. Therefore, we use regression as an optional post-processing technique. When the error between the interpolation-based pROM and the regression-based pROM is not sufficiently small, we discard the regression-based pROM and just use the interpolation-based pROM.
\item The proposed method only works when Assumption~\ref{assump:pole_change} holds. When $A^{(i)}$ in~\eqref{eq:para_ROM} is non-normal, its poles can be highly sensitive to perturbations of the parameters. We observed this phenomenon when we build order-22 Krylov-type ROMs for the branchline coupler model.
\item The proposed method can be easily extended to multiple-input multiple-output systems using the results in~\cite{YFBE19}.
\item Stability can easily be  preserved since our method deals directly with the poles. It is apparent that using linear interpolation, the stability will automatically be preserved since the linear interpolation of two poles in the left half-plane still lies in the left half-plane. Using other interpolation methods, although the stability is generally not automatically preserved, we can check whether the maximal real part of the interpolated (regressed) pole location is negative. If it is not, we can simply try another interpolation (regression) method and check it again. In the worst case, we can resort to linear interpolation for the parameter interval on which stability is difficult to be preserved with other interpolation/regression methods.
\end{itemize}

\section{Conclusions}\label{sec:con}
An adaptive parametric reduced-order modeling method based on interpolation/regression of adaptively built ROMs is proposed.
For correct interpolation/regression, we convert all ROMs to the pole-residue realization and propose to use a branch and bound method to conduct pole-matching efficiently. To explore the parameter range efficiently and accurately, we propose to use a predictor-corrector technique and an adaptive refinement strategy. A regression method is proposed to further reduce data storage. The numerical results verify the high accuracy and robustness of the proposed method.

\ifCLASSOPTIONcaptionsoff
  \newpage
\fi



%

\bibliography{/Users/yao/tex/bib/all,/Users/yao/tex/bib/mor} 
\bibliographystyle{IEEEtran}

\end{document}